\documentclass[11pt,reqno]{amsart}
\usepackage{amssymb,mathrsfs}
\usepackage{enumitem}
\usepackage{graphicx}
\usepackage{tikz}
\usepackage[bookmarksopen=true]{hyperref}
\hypersetup{
	colorlinks = false,
	 hidelinks,
}
\usetikzlibrary{arrows,shapes,positioning,calc}
\title[Adjoint orbits of the Jacobi Group]{\large Adjoint orbits of the Jacobi Group}

\usepackage{MnSymbol}

\begin{document}

\author{Yong-Jae Kwon}
\author{Jae-Hyun Yang}
\address{Yong-Jae Kwon\newline \indent Department of Mathematics, Inha University,Incheon 22212, Korea}
\email{22151060@inha.edu}
\address{Jae-Hyun Yang\newline \indent Department of Mathematics, Inha University,Incheon 22212, Korea}
\email{jhyang@inha.ac.kr}
\newtheorem{theorem}{Theorem}[section]
\newtheorem{corollary}[theorem]{Corollary}
\newtheorem{lemma}[theorem]{Lemma}
\newtheorem{proposition}[theorem]{Proposition}
\newtheorem{remark}[theorem]{Remark}
\theoremstyle{definition}
\newtheorem{definition}[theorem]{Definition}

\renewcommand{\theequation}{\thesection.\arabic{equation}}
\renewcommand{\thetheorem}{\thesection.\arabic{theorem}}
\renewcommand{\thelemma}{\thesection.\arabic{lemma}}
\newcommand{\BR}{\mathbb R}
\newcommand{\BQ}{\mathbb Q}
\newcommand{\BT}{\mathbb T}
\newcommand{\BM}{\mathbb M}
\newcommand{\bn}{\bf n}
\def\charf {\mbox{{\text 1}\kern-.24em {\text l}}}
\newcommand{\BC}{\mathbb C}
\newcommand{\BZ}{\mathbb Z}

\newcommand{\HG}{{\mathscr H}_g}
\newcommand{\GG}{\G_g^{\star}}
\newcommand{\XG}{{\mathscr X}_\BR^g}
\newcommand{\XGI}{{\mathscr X}_{(\lambda,i)}^g}

\newcommand{\OHG}{\overline{{\mathscr H}_g} }
\newcommand{\OXG}{\overline{{\mathscr X}_\BR^g} }

\newcommand\Mg{{\mathcal M}_g}
\newcommand\Rg{{\mathfrak R}_g}
\newcommand\PG{{\mathcal P}_g}

\thanks{\noindent{2010 Mathematics Subject Classification:} Primary 14L40,\;14L35; Secondary 22Exx.\\
\indent Keywords and phrases: the Jacobi group,
nilpotent orbits, the Kostant-Sekiguchi correspondence.  \\
\indent This work was supported by INHA UNIVERSITY Research Grant.}


\begin{abstract}
	In this article, we study adjoint orbits of the Jacobi group, and in particular describe nilpotent orbits explicitly.
\end{abstract}

\maketitle

\newcommand\tr{\triangleright}
\newcommand\al{\alpha}
\newcommand\be{\beta}
\newcommand\g{\gamma}
\newcommand\gh{\Cal G^J}
\newcommand\G{\Gamma}
\newcommand\de{\delta}
\newcommand\e{\epsilon}
\newcommand\z{\zeta}
\newcommand\vth{\vartheta}
\newcommand\vp{\varphi}
\newcommand\om{\omega}
\newcommand\p{\pi}
\newcommand\la{\lambda}
\newcommand\lb{\lbrace}
\newcommand\lk{\lbrack}
\newcommand\rb{\rbrace}
\newcommand\rk{\rbrack}
\newcommand\s{\sigma}
\newcommand\w{\wedge}
\newcommand\fgj{{\frak g}^J}
\newcommand\lrt{\longrightarrow}
\newcommand\lmt{\longmapsto}
\newcommand\lmk{(\lambda,\mu,\kappa)}
\newcommand\Om{\Omega}
\newcommand\ka{\kappa}
\newcommand\ba{\backslash}
\newcommand\ph{\phi}
\newcommand\M{{\Cal M}}
\newcommand\bA{\bold A}
\newcommand\bH{\bold H}
\newcommand\D{\Delta}

\newcommand\Hom{\text{Hom}}
\newcommand\cP{\Cal P}

\newcommand\cH{\Cal H}

\newcommand\pa{\partial}

\newcommand\pis{\pi i \sigma}
\newcommand\sd{\,\,{\vartriangleright}\kern -1.0ex{<}\,}
\newcommand\wt{\widetilde}
\newcommand\fg{\mathfrak{g}}
\newcommand\fk{\mathfrak{k}}
\newcommand\fp{\mathfrak{p}}
\newcommand\fs{\mathfrak{s}}
\newcommand\fh{\mathfrak{h}}
\newcommand\Cal{\mathcal}

\newcommand\fn{{\mathfrak{n}}}
\newcommand\fa{{\frak a}}
\newcommand\fm{{\frak m}}
\newcommand\fq{{\frak q}}
\newcommand\CP{{\mathcal P}_g}
\newcommand\Hgh{{\mathbb H}_g \times {\mathbb C}^{(h,g)}}
\newcommand\BD{\mathbb D}
\newcommand\BH{\mathbb H}
\newcommand\CCF{{\mathcal F}_g}
\newcommand\CM{{\mathcal M}}
\newcommand\Ggh{\Gamma_{g,h}}
\newcommand\Chg{{\mathbb C}^{(h,g)}}
\newcommand\Yd{{{\partial}\over {\partial Y}}}
\newcommand\Vd{{{\partial}\over {\partial V}}}

\newcommand\Ys{Y^{\ast}}
\newcommand\Vs{V^{\ast}}
\newcommand\LO{L_{\Omega}}
\newcommand\fac{{\frak a}_{\mathbb C}^{\ast}}

\newcommand\OW{\overline{W}}
\newcommand\OP{\overline{P}}
\newcommand\OQ{\overline{Q}}
\newcommand\Dg{{\mathbb D}_g}
\newcommand\Hg{{\mathbb H}_g}

\newcommand\OBD{ {\overline{\BD}}_g }

\newcommand\La{\Lambda}
\newcommand\FA{\mathfrak A}
\newcommand\FL{\mathfrak L}
\newcommand\FT{\mathfrak T}

\newcommand\POB{ {{\partial}\over {\partial{\overline \Omega}}} }
\newcommand\PZB{ {{\partial}\over {\partial{\overline Z}}} }
\newcommand\PX{ {{\partial}\over{\partial X}} }
\newcommand\PY{ {{\partial}\over {\partial Y}} }
\newcommand\PU{ {{\partial}\over{\partial U}} }
\newcommand\PV{ {{\partial}\over{\partial V}} }
\newcommand\PO{ {{\partial}\over{\partial \Omega}} }
\newcommand\PZ{ {{\partial}\over{\partial Z}} }
\newcommand\PW{ {{\partial}\over{\partial W}} }
\newcommand\PWB{ {{\partial}\over {\partial{\overline W}}} }
\newcommand\OVW{\overline W}

\newcommand\PR{{\mathcal P}_g\times {\mathbb R}^{(h,g)}}
\newcommand\Rmn{{\mathbb R}^{(h,g)}}
\newcommand\Gnm{GL_{g,h}}

\newcommand\fgc{{\frak g}_{\mathbb C}}
\newcommand\fkc{{\frak k}_{\BC}}
\newcommand\fpc{{\frak p}_{\BC}}

\begin{section}{{\large\bf Introduction}}
\vskip 0.3cm
It is known that if $G$ is a real reductive Lie group, there are only finitely
many nilpotent orbits and that there is the so-called Kostant-Sekiguchi correspondence
between the set of all adjoint nilpotent $G$-orbits in the Lie algebra $\mathfrak{g}$ of $G$
and the set of all $K_\mathbb{C}$-orbits in $\mathfrak{p}_\mathbb{C}$,
where $K_\mathbb{C}$ is the complexification of a maximal compact subgroup $K$
of $G$ and $\mathfrak{g}_\mathbb{C}=\mathfrak{k}_\mathbb{C}+\mathfrak{p}_\mathbb{C}$ is
the Cartan decomposition of the complexification of $\mathfrak{g}_\mathbb{C}$ of $\mathfrak{g}$\,(cf.\,\cite{S-V, Se, Ve,Vo}).

\vskip 3mm
In this paper, we consider the Jacobi group
\begin{equation*}
G^J=SL(2,\mathbb{R})\ltimes H_{\BR},
\end{equation*}
where is the semidirect product of the special linear group $SL(2,\BR)$ and the
three dimensional Heisenberg group $H_{\BR}$. Here
\begin{equation*}
H_{\BR}=\left\{\,(\la,\mu,\kappa)\,\bigg| \ \la,\mu,\kappa\in
\BR\, \right\}
\end{equation*}
is the 2-step nilpotent Lie group with the following multiplication law
\begin{equation*}
(\la,\mu,\kappa)\circ
(\la',\mu',\kappa')=(\la+\la',\mu+\mu',\kappa+
\kappa'+\la\mu'-\la'\mu).
\end{equation*}
The Jacobi group $G^J$ is a {\sf non-reductive} Lie group endowed with the
following multiplication
\begin{equation}
(M,(\lambda,\mu,\kappa))\cdot(M',(\lambda',\mu',\kappa')) =\,
(MM',(\tilde{\lambda}+\lambda',\tilde{\mu}+ \mu',
\kappa+\kappa'+\tilde{\lambda}\mu' -\la'\tilde{\mu})).
\end{equation}
Let $\mathbb{H}$ be the Poincar\'{e} upper half plane. Then $G^J$ acts on
the Siegel-Jacobi space $\mathbb{H}\times\mathbb{C}$ transitively by
\begin{equation}\label{1eq05}
(M,(\la,\mu,\kappa))\cdot(\tau,z)=(M<\tau>,(z+\la
\tau+\mu)(c\tau+d)^{-1}),
\end{equation}
where $M=\begin{pmatrix} a&b\\
c&d\end{pmatrix} \in SL(2,\BR),\ (\lambda,\mu, \kappa)\in
H_{\BR}$ and $(\tau,z)\in  \mathbb{H}\times \BC$ and
\begin{equation*}
M<\tau>=(a\tau+b)(c\tau+d)^{-1}.
\end{equation*}
It is easily seen that
the stabilizer $K^J$ of $G^J$ at $(i,0)$ under the action (\ref{1eq05}) is
given by
\begin{equation}
K^J=\left\{(k,(0,0,\kappa))\vert\ k\in K,\
\kappa\in\BR\,\right\}\cong K\times\BR,
\end{equation}
where $K=SO(2)$ is a maximal compact subgroup of $SL(2,\mathbb{R})$. Then the Siegel-Jacobi space
$\mathbb{H}\times\mathbb{C}$ is biholomorphic to the Hermitian-K\"{a}hler homogeneous
space $G^J/K^J$ via
\begin{equation}
\begin{array}{ccc}
	 G^J/K^J& \longrightarrow &\hspace{-60pt}\mathbb{H}\times\mathbb{C}\\
	\rotatebox{90}{\ensuremath{\in}} &  & \hspace{-60pt}\rotatebox{90}{\ensuremath{\in}}\\
	gK^J & \longmapsto &g\cdot(i,0),\qquad g\in G^J.
\end{array}
\end{equation}
Therefore the Jacobi group $G^J$ plays an important role in number theory (e.g. theory of Jacobi forms)\,\cite{EZ, YJH1, YJH2, YJH3, YJH4, YJH5, YJH6, YJH14, YJH16, YY, Z},
algebraic geometry\,\cite{YJH8, YJH9, YJH16, YJH18}, complex geometry\,\cite{YJH10, YJH11, YJH12, YJH13, YJH16}, representation theory\,\cite{BS, YJH15, YJH17} and
mathematical physics\,\cite{B, M}.
\par In this paper, we study the adjoint orbits of $G^J$, and in particular, we calculate
the adjoint nilpotent orbits of $G^J$ explicitly.
We show that unlike the case of a reductive Lie group,
there are uncountably many nilpotent $G^J$-orbits.
\par This paper is organized as follows. In Section 2,  we review the Kostant-Sekiguchi correspondence for a reductive real Lie group and adjoint orbits of $SL(2,\BR)$. In Section 3, we study the adjoint orbits of $G^J$ in the Lie
algebra $\mathfrak{g}^J$. We describe the set of nilpotent orbits of $G^J$ and the set of
nilpotent orbits of $K_\mathbb{C}^J$ in $\mathfrak{p}_\mathbb{C}^J$ explicitly. Here $K_\mathbb{C}^J$ is the complexification of $K^J$ and $\mathfrak{g}_\mathbb{C}^J=\mathfrak{k}_\mathbb{C}^J+\mathfrak{p}_\mathbb{C}^J$ is the decomposition of the complexification $\mathfrak{g}_\mathbb{C}^J$ of $\mathfrak{g}^J$.
\vskip 0.5cm\noindent
{\bf Notations:}
We denote by $\BZ,\,\BR$ and $\BC$ the ring of integers, the field
of real numbers, and the field of complex numbers respectively. We
denote by $\BR^{\times}$ and $\BC^{\times}$ the set of nonzero
real numbers and the set of nonzero complex numbers respectively.
We denote by $\BZ^+$\,(resp.\,$\BZ_{\geq 0}$) the set of all
positive\,(resp. nonnegative) integers, by $F^{(k,l)}$ the set of
all $k\times l$ matrices with entries in a commutative ring $F$.
For any $M\in F^{(k,l)},\ ^t\!M$ denotes the transpose matrix of
$M$. We denote the identity matrix of degree $n$ by $I_n$.

\end{section}

\vskip 1.5cm

\section{{\large\bf The Kostant-Sekiguchi Correspondence}}
\setcounter{equation}{0}
\vskip 2mm
In this section, we review the Kostant-Sekiguchi correspondence for a reductive real Lie group and adjoint orbits of $SL(2,\BR)$\,(cf.\,\cite{K, L, S-V, Se, Ve, Vo}).
Let $G$ be a real reductive group with Lie algebra $\mathfrak{g}$, and let
$K$ be a maximal compact subgroup of $G$ with Lie algebra $\mathfrak{k}$. Let $\mathfrak{g} = \mathfrak{k}\oplus \mathfrak{p}$ be a Cartan decomposition of $\mathfrak{g}$ with the assciated Cartan involution $\theta$. Let $\mathfrak{g}_\mathbb{C}=\mathfrak{k}_\mathbb{C}\oplus\mathfrak{p}_\mathbb{C}$ denote the complexification of $\mathfrak{g}$, and let $\sigma$ be the associated complex conjugation. Let $G_\mathbb{C}$ and $K_\mathbb{C}$ denote the complexifications of $G$ and $K$ with the Lie algebras $\mathfrak{g}_\mathbb{C}$ and $\mathfrak{k}_\mathbb{C}$, respectively. \\

J.\,Sekiguchi and B.\,Kostant established a bijection between the set of all nilpotent $G$-orbits in $\mathfrak{g}$ and the set of all nilpotent ${K_\mathbb{C}}$-orbit in $\mathfrak{p}_\mathbb{C}$.

\vskip 0.3cm\noindent
\begin{definition}Let $L$ denote $\mathfrak{g}$ or $\mathfrak{g}_\mathbb{C}$.
		\begin{enumerate}[itemsep=6pt]
	\item An ordered triple $\{Z_1,Z_2,Z_3\}$ of elements in $L$ is said to be an \textit{$\mathfrak{sl}_2$-triple} if
	\begin{equation}
	[Z_1,Z_2]=2Z_2\,,\qquad[Z_1,Z_3]=-2Z_3\,,\qquad[Z_2,Z_3]=Z_1\,.
	\end{equation}
	\item Two $\mathfrak{sl}_2$-triples $\{Z_1,Z_2,Z_3\}$ and $\{Z'_1,Z'_2,Z'_3\}$ in $L$ are said to be \textit{conjugate under a subgroup $W$} of $L$ if there exists an element $w \in W$ such that $Z_i=w\cdot
	Z'_i\,\big(=w\, Z'_i \,w^{-1}\big)$ for $i=1,2,3$.
	\end{enumerate}
\end{definition}
\vskip 0.3cm\noindent

To describe the Kostant-Sekiguchi correspondence, it is necessary to consider the following classes of $\mathfrak{sl}_2$-triples.
\vskip 0.3cm

\begin{definition}[Kostant-Sekiguchi triples] \hfill
	\begin{enumerate}[itemsep=6pt]
		\item An $\mathfrak{sl}_2$-triple $\{H,E,F\}$ in $\mathfrak{g}$ is said to be a \textit{KS-triple in $\mathfrak{g}$} if $\theta(E)=-F$.
		\item An $\mathfrak{sl}_2$-triple $\{x,e,f\}$ in $\mathfrak{g}_\mathbb{C}$ is said to be a \textit{normal} if $x \in \mathfrak{k}_\mathbb{C}$ and $e,f \in \mathfrak{p}_\mathbb{C}$.
		\item A normal $\mathfrak{sl}_2$-triple $\{x,e,f\}$ in $\mathfrak{g}_\mathbb{C}$ is said to be a \textit{KS-triple in $\mathfrak{g}_\mathbb{C}$} if $f=\sigma(e)$.
	\end{enumerate}
\end{definition}

\vskip 0.3cm\noindent
\begin{theorem}\label{2theorem01}
	Let $G$ be a real reductive group with Lie algebra $\mathfrak{g}$, and let
	$K$ be a maximal compact subgroup of $G$ with Lie algebra $\mathfrak{k}$. Let $\mathfrak{g}_\mathbb{C}=\mathfrak{k}_\mathbb{C}\oplus\mathfrak{p}_\mathbb{C}$ denote the complexification of $\mathfrak{g}$.  Let $K_\mathbb{C}$ denote the complexification of $K$ with the Lie algebra $\mathfrak{k}_\mathbb{C}$. The following sets (1)-(6) are in natural one-to-one correspondence:
\begin{enumerate}[itemsep=6pt]
	\item Nilpotent $G$-orbits in $\mathfrak{g}.$
	\item $G$-conjugacy classes of $\mathfrak{sl}_2$-triples in $\mathfrak{g}.$
	\item $K$-conjugacy classes of KS-triples in $\mathfrak{g}.$
	\item $K$-conjugacy classes of KS-triples in $\mathfrak{g}_\mathbb{C}.$
	\item $K_\mathbb{C}$-conjugacy classes of normal $\mathfrak{sl}_2$-triples in $\mathfrak{g}_\mathbb{C}.$
	\item Nilpotent $K_\mathbb{C}$-orbits in $\mathfrak{p}_\mathbb{C}$.
\end{enumerate}
The correspondence between (1) and (6) is the Kostant-Sekiguchi correspondence.
\end{theorem}
We refer to \cite{K,KR,Se,YJH8.5} for more details on Theorem \ref{2theorem01}.
\vskip 0.5cm
\begin{remark} With the notations as in Theorem \ref{2theorem01}, M.Vergne \cite{Ve} proved that if $\mathcal{O}$ is a real nilpotent orbit in $\mathfrak{g}$, then there exists a canonical $K$-equivariant diffeomorphism of $\mathcal{O}$ onto the nilpotent $K_\mathbb{C}$-orbit in $\mathfrak{p}_\mathbb{C}$ associated to $\mathcal{O}$ via the Kostant-Sekiguchi correspondence. (cf.\;\cite{Vo} p.206)
\end{remark}
\vskip 0.5cm
\noindent
{\large \bf Example.}
We let $G=SL(2,\BR)$ and let $K=SO(2)$ be a maximal
compact subgroup of $G$. The Lie algebra $\fg$ of $G$ is given by

\begin{equation*}
\fg=\left\{ \begin{pmatrix} x & \ y \\ z & -x\end{pmatrix}\bigg| \ x,y,z\in
\BR\,\right\}.
\end{equation*}
We put
\begin{equation*}
X=\begin{pmatrix} 1 & \ 0 \\ 0 & -1
\end{pmatrix},\quad
Y=\begin{pmatrix} 0 & 1 \\ 1 & 0
\end{pmatrix},\quad
Z=\begin{pmatrix} \ 0 & 1 \\ -1 & 0
\end{pmatrix}.
\end{equation*}
Then the set $\left\{ X,Y,Z\right\}$ forms a basis for $\fg$. We
define an element $F(x,y,z)\in \fg $ by
\begin{equation}
F(x,y,z):=xX+yY+zZ=\begin{pmatrix} x & y+z \\ y-z & -x
\end{pmatrix}.
\end{equation}
Then we have the relations
\begin{equation}
X^2+Y^2-Z^2=3I_2,\quad [X,Y]=2Z,\quad
[X,Z]=2Y,\quad [Y,Z]=-2X.
\end{equation}
It is easy to see that $X$ and
$Y$ are hyperbolic elements and $Z$ is an elliptic element. For a
nonzero real number $\alpha$, the $G$-orbit of $\alpha X$ is represented
by the one-sheeted hyperboloid
\begin{equation}
x^2+y^2-z^2=\alpha^2. \label{3eq05}
\end{equation}
The $G$-orbit of $\al Y\,(\alpha\in \BR^{\times})$ is also represented by
the hyperboloid (\ref{3eq05}). The $G$-orbit of $\al Z\,(\al\in
\BR^{\times})$ is represented by two-sheeted hyperboloids
\begin{equation}
x^2+y^2-z^2=-\al^2.\label{3eq06}
\end{equation}
Since
$$F(x,y,z)^2=(x^2+y^2-z^2)\cdot I_2,$$ we have for any
$k\in\BZ^+,$ $$F(x,y,z)^{2k}=(x^2+y^2-z^2)^k\cdot I_2.$$
Thus we
see that $F(x,y,z)$ is nilpotent if and only if $x^2+y^2-z^2=0.$
Therefore the set ${\Cal N}_{\BR}$ of all nilpotent elements in
$\fg$ is given by

\begin{equation}
{\Cal N}_{\BR}=\left\{ F(x,y,z)=\begin{pmatrix} x &
y+z\\ y-z & -x\end{pmatrix} \bigg|\ x^2+y^2-z^2=0\right\}.
\end{equation}

We put
\begin{equation}
S={\frac 12}(Y+Z)=\begin{pmatrix} 0 & 1\\ 0 & 0\end{pmatrix},\quad
T= {\frac 12}(Y-Z)=\begin{pmatrix} 0 & 0\\ 1 & 0\end{pmatrix}.
\end{equation}
Obviously $S$ and $T$ are nilpotent elements in ${\Cal N}_{\BR}$
and they satisfy
\begin{equation}
[X,S]=2S\,,\qquad[X,T]=-2T\,,\qquad[S,T]=X\,\label{3eq08}
\end{equation}
and
\begin{equation}
\theta(X)=-X\,,\qquad\theta(S)=-T,\qquad\theta(T)=-S. \label{3eq09}
\end{equation}
Here $\theta$ is the Cartan involution defined by  $\theta(g)= -{}^tg$ for $g$ in $\mathfrak{g}$. \\
According to equation (\ref{3eq08}) and (\ref{3eq09}), $\{X,\,S,\,T\}$ and $\{-X,\,-S,\,-T\}$ are KS-triples in $\mathfrak{g}$.
\par The $G$-orbit of $\al S\,(\al\in \BR^{\times})$ is represented
by the cone
\begin{equation}
x^2+y^2-z^2=0,\quad (x,y,z)\neq (0,0,0)
\end{equation}
depending on the sign of $\al$.
\vskip 0.3cm
If $\al > 0,$ the $G$-orbit of $\al S$ is characterized by the one-sheeted
cone
\begin{equation}
x^2+y^2-z^2=0,\quad z>0.\label{3eq10}
\end{equation}

If $\al < 0,$ the
$G$-orbit of $\al S$ is characterized by the one-sheeted cone

\begin{equation}
x^2+y^2-z^2=0,\quad z<0.\label{3eq11}
\end{equation}
The $G$-orbits of $\al
T\,(\al>0)$ are characterized by the one-sheeted cone (\ref{3eq11}) and
the $G$-orbits of $\al T\,(\al<0)$ are characterized by the
one-sheeted cone (\ref{3eq10}).

\vskip 0.3cm
We define
the $G$-orbits ${\Cal N}_{\BR}^+$ and ${\Cal N}_{\BR}^-$ by
\begin{equation}
{\Cal N}_{\BR}^+=G\cdot S\quad \text{and}\quad {\Cal N}_{\BR}^-=G\cdot T.
\end{equation}
Then we obtain

\begin{equation}
{\Cal N}_{\BR}={\Cal N}_{\BR}^+\cup \left\{ 0\right\}\cup {\Cal N}_{\BR}^-.\label{3eq13}
\end{equation}
According to (\ref{3eq05}),\,(\ref{3eq06}) and (\ref{3eq13}), we
see that there are infinitely many hyperbolic orbits and elliptic
orbits, and on the other hand there are only three nilpotent
orbits in $\fg$.

\vskip 0.3cm
Let $$K_{\BC}=SO(2,\BC)=\left\{ \begin{pmatrix} \ a & b\\ -b & a\end{pmatrix}
\bigg|\ a^2+b^2=1,\ a,b\in \BC\,\right\}$$
be the complexification
of $K$. The complexification $\fg_{\BC}$ of $\fg$ has the Cartan
decomposition
$$\fgc=\fkc+\fpc,$$
where
$$\fkc=\left\{ \begin{pmatrix} \ 0
& z\\ -z & 0\end{pmatrix} \bigg|\ z\in \BC\,\right\}$$
and
$$\fpc=\left\{ \begin{pmatrix} x & \ y\\ y & -x\end{pmatrix} \bigg|\ x,y\in
\BC\,\right\}.$$
The set ${\Cal N}_{\theta}$ of all nilpotent
elements in $\fpc$ is given by
\begin{equation}
{\Cal N}_{\theta}=\left\{
\begin{pmatrix} x & \ y\\ y & -x\end{pmatrix}\in \fpc \;\bigg|\
x^2+y^2=0\,\right\}\subset \fpc.
\end{equation}
 We note that
$K_{\BC}$ acts on ${\Cal N}_{\theta}$.
\vskip 0.3cm
\par We put

\begin{equation}
H_{\theta}=\begin{pmatrix} \ 0 & -i \\ i & 0\end{pmatrix},\quad
X_{\theta}={\frac 12} \begin{pmatrix} -i & \ 1 \\ 1 & i\end{pmatrix},\quad
Y_{\theta}={\frac 12}\begin{pmatrix} \ i & 1 \\ 1 &-i\end{pmatrix}.
\end{equation}
Then they satisfy
\begin{equation}
H_{\theta}\in \fkc,\quad X_{\theta},\,Y_{\theta}\in \fpc \label{3eq18}
\end{equation}
,
\begin{equation}
[H_{\theta},Y_{\theta}]=2Y_{\theta},\quad
[H_{\theta},X_{\theta}]=-2X_{\theta},\quad
[Y_{\theta},X_{\theta}]=H_{\theta}.\label{3eq19}
\end{equation}
and
\begin{equation}
\s_0(H_{\theta})=-H_{\theta},\quad
\s_0(X_{\theta})=Y_{\theta},\quad
\s_0(Y_{\theta})=X_{\theta},\label{3eq20}
\end{equation}
where $\s_0$ denotes the
complex conjugation on $\fgc.$
\vskip 0.3cm
\par According to equation (\ref{3eq18}), (\ref{3eq19}) and (\ref{3eq20}), $\{H_{\theta},\,Y_{\theta},\,X_{\theta}\}$ and $\{-H_{\theta},\,-Y_{\theta},\,-X_{\theta}\}$ are KS-triples in $\mathfrak{g}_\mathbb{C}$.
Moreover, two KS-triples $\{X,\,S,\,T\}$ and $\left\{H_{\theta},\,Y_{\theta},\,X_{\theta}\right\}$ satisfy the following conditions:
\begin{equation}
H_{\theta}=i(S-T)\,,\quad Y_{\theta}=\frac{1}{2}(S+T+iX)\,,\quad X_{\theta}=\frac{1}{2}(S+T-iX).
\end{equation}
\par Now we define the
$K_{\BC}$-orbits ${\Cal N}_{\theta}^+$ and ${\Cal N}_{\theta}^-$
by
\begin{equation}
{\Cal N}_{\theta}^+=K_{\BC}\cdot X_{\theta}\quad
\text{and}\quad {\Cal N}_{\theta}^-=K_{\BC}\cdot
Y_{\theta}.
\end{equation}
Then we see that
\begin{equation}
{\Cal N}_{\theta}={\Cal N}_{\theta}^+\cup \left\{ 0\right\}\cup {\Cal N}_{\theta}^-.
\end{equation}

The $K_{\BC}$-orbit ${\Cal N}_{\theta}^+$
is characterized by the straight line

\begin{equation}
y=ix,\quad x\in\BC-\left\{ 0\right\}.
\end{equation}

On the other hand, the $K_{\BC}$-orbit ${\Cal
N}_{\theta}^-$ is characterized by the straight line
\begin{equation}
y=-ix,\quad x\in\BC-\left\{ 0\right\}.
\end{equation}

It is easily seen that the
$K_{\BC}$-orbits of $\al H_{\theta}\,(\al \in\BC^{\times})$ are
represented by complex hyperboloids and that there are infinitely
many hyperbolic and elliptic orbits in $\fgc.$ However there are
only three nilpotent orbits in $\fpc$ which are
${\Cal N}_{\theta}^+,\ \left\{ 0\right\}$ and ${\Cal N}_{\theta}^-.$

\vskip 0.3cm
The Kostant-Sekiguchi correspondence between the
$G$-nilpotent orbits in ${\Cal N}_{\BR}$ and the
$K_{\BC}$-nilpotent orbits in ${\Cal N}_{\theta}$ is given by
\begin{equation}
{\Cal N}_{\BR}^+\mapsto {\Cal N}_{\theta}^-,\quad \left\{
0\right\}\mapsto \left\{ 0\right\}, \quad {\Cal N}_{\BR}^-\mapsto
{\Cal N}_{\theta}^+.
\end{equation}


\vskip 1.5cm

\begin{section}{{\large\bf  Adjoint Orbits of the Jacobi Group}}
\setcounter{equation}{0}
In this section, we compute the adjoint orbits for the Jacobi group.
We observe that the Jacobi group $G^J$ is embedded in the symplectic group
$Sp(4,\BR)$ via
\begin{equation}
(M,(\la,\mu,\kappa))\mapsto \begin{pmatrix} a & 0 & b &
a \mu-b\la
\\ \la & 1 & \mu & \kappa
\\ c & 0 & d & c\mu-d\la \\ 0 & 0 & 0 & 1 \end{pmatrix},\label{4eq06}
\end{equation}
where $M=\begin{pmatrix} a & b\\ c & d\end{pmatrix}\in SL(2,\BR).$ The Lie
algebra $\fg^J$ of $G^J$ is given by
\begin{equation}
\fg^J=\left\{ (X,(p,q,r))\,|\ X\in \fg,\ p,q,r\in
\BR\,\right\}
\end{equation}
with the bracket
\begin{equation}
[(X_1,(p_1,q_1,r_1)),\,(X_2,(p_2,q_2,r_2))]=({\tilde X},({\tilde
p},{\tilde q},{\tilde r})),
\end{equation}
where
$$X_1=\begin{pmatrix} x_1 & \
y_1\\ z_1 & -x_1\end{pmatrix},\quad X_2=\begin{pmatrix} x_2 & \ y_2\\ z_2 &
-x_2\end{pmatrix}\in \mathfrak{sl}(2,\mathbb{R})$$
and
\begin{eqnarray*}
{\tilde X}&=& X_1X_2-X_2X_1,\\
{\tilde p}&=& p_1x_2+q_1z_2-p_2x_1-q_2z_1,\\
{\tilde q}&=& q_2x_1+p_1y_2-q_1x_2-p_2y_1,\\
{\tilde r}&=& 2(p_1q_2-p_2q_1).
\end{eqnarray*}
Indeed, an element $(X,(p,q,r))$ in $\fg^J$ with $X=\begin{pmatrix} x &
y+z\\ y-z & -x\end{pmatrix} \in \mathfrak{sl}(2,\mathbb{R})$ may be identified with the
matrix
\begin{equation}
G(x,y,z,p,q,r):=\begin{pmatrix} x & 0 & y+z & q \\ p & 0 & q & r
\\ y-z & 0 & -x & -p \\ 0 & 0 & 0 & 0 \end{pmatrix}\label{4eq09}
\end{equation}
in the Lie algebra $\frak{sp}(4,\BR)$ of $Sp(4,\BR).$

\vskip 0.3cm\noindent
\begin{lemma}\label{4lem01}
If $G(x,y,z,p,q,r)$ is an element in $\fg^J$
given by (\ref{4eq09}), then for a positive integer $k\in \BZ^+$,
\begin{equation}
G(x,y,z,p,q,r)^{2k}=(x^2+y^2-z^2)^{k-1}G(x,y,z,p,q,r)^2.\label{4eq10}
\end{equation}
\end{lemma}
\begin{proof}
	By the Cayley-Hamilton theorem or a direct computation, we obtain
	\begin{equation}
	G(x,y,z,p,q,r)^4=(x^2+y^2-z^2)G(x,y,z,p,q,r)^2.\label{4eq11}
	\end{equation}
	The formula
	(\ref{4eq10}) follows immediately from (\ref{4eq11}).
\end{proof}

\vskip 0.5cm
According to Lemma \ref{4lem01}, the set ${\Cal N}_{\BR}^J$ of all
nilpotent elements in $\fg^J$ is given by
\begin{equation}\label{4eq07}
{\Cal N}_{\BR}^J=\left\{\,G(x,y,z,p,q,r)\in \fgj\vert\
x^2+y^2-z^2=0\,\right\}.
\end{equation}
We have the adjoint action of
$G^J$ on $\fg^J$ given by
\begin{equation}
g\cdot X=Ad(g)X=gXg^{-1},\quad g\in G^J,\ X\in {\Cal N}_{\BR}^J.
\end{equation}
According to (\ref{4eq06}), we may
write $g=(M,(\la,\mu,\kappa))\in G^J$ with $M=\begin{pmatrix} a
&b\\c&d\end{pmatrix} \in SL(2,\mathbb{R})$ as
\begin{equation}
g=\begin{pmatrix} a & 0 & b & a\mu-b\la
\\ \la & 1 & \mu & \kappa
\\ c & 0 & d & c\mu-d\la \\ 0 & 0 & 0 & 1 \end{pmatrix}.
\end{equation}
Then the inverse of $g$ is given by
\begin{equation}
g^{-1}=\begin{pmatrix} d & 0 & -b &
-\mu
\\ c\mu-d\la & 1 & b\la-a\mu & -\kappa
\\ -c & 0 & a & \la \\ 0 & 0 & 0 & 1 \end{pmatrix}.
\end{equation}

\vskip 0.5cm
\begin{lemma}\label{4lem02}
If $g=\left( \begin{pmatrix} a & b
\\c&d\end{pmatrix}, (\la,\mu,\kappa)\right)$ is an element of $G^J$, then the
action of $g$ on $G(x,y,z,p,q,r)$ is given by

\begin{equation}
g\cdot
G(x,y,z,p,q,r)=G({\tilde x},{\tilde y},{\tilde z},{\tilde
p},{\tilde q},{\tilde r}),
\end{equation}
where
\begin{eqnarray*}
{\tilde x}&=&(ad+bc)x-ac(y+z)+bd(y-z),\\
{\tilde y}+{\tilde z}&=& -2abx+a^2(y+z)-b^2(y-z),\\
{\tilde y}-{\tilde z}&=& 2cdx-c^2(y+z)+d^2(y-z),\\
{\tilde p}&=&d\left\{\la x+\mu(y-z)+p\right\}+c\left\{\mu x-\la (y+z)-q\right\},\\
{\tilde q}&=&-b\left\{\la x+\mu(y-z)+p\right\}-a\left\{\mu x-\la (y+z)-q\right\},\\
{\tilde r}&=&-2\la\mu x +\la^2(y+z)-\mu^2(y-z)-2p\mu+2q\la+r.
\end{eqnarray*}
In particular,
${\Cal N}_{\BR}^J$ is stable under the action of $G^J$.
\end{lemma}
\begin{proof}
The proof of the first part follows from a direct
computation. Let $G(x,y,z,p,q,r)$ be an element of ${\Cal
	N}_{\BR}^J$. Since
$${\tilde x}^2+{\tilde y}^2-{\tilde
	z}^2=(ad-bc)^2(x^2+y^2-z^2)=0,$$
we see that $g\cdot {\Cal
	N}_{\BR}^J\subset {\Cal N}_{\BR}^J$ for all $g\in G^J.$
\end{proof}
\vskip 0.3cm
We set
\begin{eqnarray*}
X^J&=&G(1,0,0,0,0,0),\\
Y^J&=&G(0,1,0,0,0,0),\\
Z^J&=&G(0,0,1,0,0,0),\\
P^J&=&G(0,0,0,1,0,0),\\
Q^J&=&G(0,0,0,0,1,0),\\
R^J&=&G(0,0,0,0,0,1).
\end{eqnarray*}

\vskip 0.3cm\noindent
Obviously the set
$\left\{\,X^J,\,Y^J,\,Z^J,\,P^J,\,Q^J,\,R^J\,\right\}$ forms a
basis for $\fgj$. So we have
$$G(x,y,z,p,q,r)=xX^J+yY^J+zZ^J+pP^J+qQ^J+rR^J.$$
We note that
$X^J$ and $Y^J$ are hyperbolic elements, $Z^J$ is an elliptic
element and $P^J,\,Q^J,\,R^J$ are nilpotent elements.

\vskip 0.3cm
Let $\Pi(G(x,y,z,p,q,r))$ denote the $G^J$-orbit of $G(x,y,z,p,q,r)$. According to Lemma \ref{4lem02}, we can get the following lemma.
\begin{lemma}\label{4lem03}
	Let $\al\in\BR$ be a fixed nonzero real number. Then
	\begin{eqnarray}
		\Pi(\alpha X^J)&=&\Pi(\alpha Y^J)\;=\;\left\lbrace G(x,y,z,p,q,r) \in {\mathfrak g}^J \;\Bigg|\; \begin{array}{c}
			x^2+y^2-z^2=\alpha^2 \\
			f(x,\,y,\,z,\,p,\,q)=\alpha^2 r
		\end{array}  \right\rbrace , \\
		\Pi(\alpha Z^J)&=&\left\lbrace G(x,y,z,p,q,r) \in {\mathfrak g}^J \;\Bigg|\; \begin{array}{c}
			x^2+y^2-z^2=-\alpha^2 \\
			f(x,\,y,\,z,\,p,\,q)=-\alpha^2 r
		\end{array}  \right\rbrace , \\
		\Pi(\alpha P^J)&=&\Pi(\alpha Q^J)\;=\;\left\lbrace G(0,0,0,p,q,r) \in {\mathfrak g}^J \;\big|\;pq\neq0   \right\rbrace , \\
		\Pi(\alpha R^J)&=&\alpha R^J ,
	\end{eqnarray}
where
\begin{equation}
	f(x,\,y,\,z,\,p,\,q)=2pqx-p^2(y+z)+q^2(y-z).
\end{equation}
\end{lemma}
\begin{proof}By Lemma \ref{4lem02},
	\begin{align*}
	\Pi(\alpha X^J)&=\left\lbrace \alpha G(x,y,z,p,q,r) \in {\mathfrak g}^J \;\Bigg|\; \begin{array}{c}
	x=ad+bc,\; y+z=-2ab,\; y-z=2cd,\\
	p=d {\lambda} + c \mu,\; q= -b {\lambda} - a \mu ,\; r= -2 \, {\lambda} \mu, \\
	ad-bc=1,\; a,b,c,d,\la,\mu \in \mathbb{R}
	\end{array}  \right\rbrace , \\
	\Pi(\alpha Y^J)&=\left\lbrace \alpha G(x,y,z,p,q,r) \in {\mathfrak g}^J \;\Bigg|\; \begin{array}{c}
	x=-ac+bd,\; y+z=a^2-b^2,\; y-z=-c^2+d^2,\\
	p=-c {\lambda} + d \mu,\; q= a {\lambda} -b \mu ,\; r= {\lambda}^2-\mu^2, \\
	ad-bc=1,\; a,b,c,d,\la,\mu \in \mathbb{R}
\end{array}  \right\rbrace , \\
	\Pi(\alpha Z^J)&=\left\lbrace \alpha G(x,y,z,p,q,r) \in {\mathfrak g}^J \;\Bigg|\; \begin{array}{c}
	x=-ac-bd,\; y+z=a^2+b^2,\; y-z=-c^2-d^2,\\
	p=-c {\lambda} - d \mu,\; q= a {\lambda} +b \mu ,\; r= {\lambda}^2+\mu^2, \\
	ad-bc=1,\; a,b,c,d,\la,\mu \in \mathbb{R}
\end{array}  \right\rbrace , \\
	\Pi(\alpha P^J)&=\left\lbrace \alpha G(0,0,0,p,q,r) \in {\mathfrak g}^J \;\big|\;
	p=d,\;q=-b,\; r=-2\mu, \;b,d,\mu \in \mathbb{R} \right\rbrace , \\
	\Pi(\alpha Q^J)&=\left\lbrace \alpha G(0,0,0,p,q,r) \in {\mathfrak g}^J \;\big|\;
p=-c,\;q=a,\; r=2{\lambda}, \;a,c,\lambda \in \mathbb{R} \right\rbrace , \\
	\Pi(\alpha R^J)&=\alpha R^J . \qedhere
	\end{align*}
\end{proof}
We define the nilpotent
elements $S^J$ and $T^J$ by
\begin{equation}
S^J={\frac 12}(Y^J+Z^J)\quad\text{and}\quad T^J={\frac 12}(Y^J-Z^J).
\end{equation}

\begin{lemma}\label{4lem06} 	Let $\al \in\BR$ be a fixed nonzero real number.
Then $\alpha G(0,1,1,p,q,r)$ and \\ $\alpha G(0,1,1,p,0,r- \frac{1}{2}q^2)$ lie in the same $G^J$-orbit in ${\mathfrak g}^J$.
\end{lemma}
\begin{proof}
	If $g=\left( I_2, (-q/2,0,0)\right)$, then the action of $g$ on $G(0,1,1,p,q,r)$ is given by
	\begin{equation*}
	g\cdot G(0,1,1,p,q,r)=G\left( 0,1,1,p,0,r-\frac{1}{2}q^2\right) .\qedhere
	\end{equation*}
\end{proof}
\begin{lemma}\label{4lem07}
	Let $\al,\beta \in\BR$ be fixed nonzero real numbers. Then
	\begin{eqnarray}
	\Pi(\alpha S^J)&=&\left\lbrace G(x,y,z,p,q,r) \in {\mathfrak g}^J \;\Bigg|\; \begin{array}{c}
	x^2+y^2-z^2=0, \; z/\alpha>0, \\
	f(x,\,y,\,z,\,p,\,q)=0,  \\
	r\textsl{ is dependant on } x,y,z,p,q.
	\end{array}  \right\rbrace , \\
	\Pi(\alpha T^J)&=&\left\lbrace G(x,y,z,p,q,r) \in {\mathfrak g}^J \;\Bigg|\; \begin{array}{c}
	x^2+y^2-z^2=0, \; z/\alpha<0, \\
	f(x,\,y,\,z,\,p,\,q)=0,  \\
	r\textsl{ is dependant on } x,y,z,p,q.
	\end{array}  \right\rbrace, \\
	&=&\Pi(\alpha (-S^J)), \nonumber \\
	\hspace{-30pt}\Pi(\alpha(S^J+\beta P^J))&=&\left\lbrace G(x,y,z,p,q,r) \in {\mathfrak g}^J \;\Bigg|\; \begin{array}{c}
	x^2+y^2-z^2=0, \; z/\alpha>0, \\
	f(x,\,y,\,z,\,p,\,q)=-\alpha^3\beta^2
	\end{array}  \right\rbrace, \\
	&=&\Pi(\alpha|\beta|^{2/3}(S^J+P^J)) \nonumber
	\end{eqnarray}
	where
	\begin{equation*}
	f(x,\,y,\,z,\,p,\,q)=2pqx-p^2(y+z)+q^2(y-z).
	\end{equation*}
\end{lemma}
\begin{proof}By Lemma \ref{4lem02},
	\begin{align*}
		\Pi(\alpha S^J)&=\left\lbrace \alpha G(x,y,z,p,q,r) \in {\mathfrak g}^J \;\Bigg|\; \begin{array}{c}
			x=-ac,\; y+z=a^2,\; y-z=-c^2,\\
			p=-c{\lambda},\; q=a{\lambda},\; r={\lambda}^2, \\
			ad-bc=1,\; a,b,c,d,\la,\mu \in \mathbb{R}
		\end{array}  \right\rbrace , \\
		\Pi(\alpha T^J)&=\left\lbrace \alpha G(x,y,z,p,q,r) \in {\mathfrak g}^J \;\Bigg|\; \begin{array}{c}
	x=bd,\; y+z=-b^2,\; y-z=d^2,\\
	p=d\mu,\; q=-b\mu,\; r=-\mu^2, \\
	ad-bc=1,\; a,b,c,d,\la,\mu \in \mathbb{R}
\end{array}  \right\rbrace , \\
		\hspace{-20pt}\Pi(\alpha (S^J+\beta P^J))&=\left\lbrace \alpha G(x,y,z,p,q,r) \in {\mathfrak g}^J \;\Bigg|\; \begin{array}{c}
	x=-ac,\; y+z=a^2,\; y-z=-c^2,\\
	p=-c{\lambda}+\beta d,\; q=a{\lambda}-\beta b,\; r={\lambda}^2-2\beta\mu, \\
	ad-bc=1,\; a,b,c,d,\la,\mu \in \mathbb{R}
\end{array}  \right\rbrace . \qedhere
	\end{align*}
\end{proof}

Now we can prove the following theorem.

\begin{theorem}\label{4thm01}We have a disjoint union
	\begin{eqnarray}
		{\mathcal{N}}_{\mathbb{R}}^J& = &\left\{0\right\} \bigcup \Pi(S^J)\bigcup \Pi(T^J)\bigcup \Pi(P^J)\bigcup\left(\bigcup_{\alpha\in\mathbb{R}^{\times}}\left\{ \Pi(\alpha
		R^J)\right\}\right)\label{4eq21}\\
		&& \bigcup \left(\bigcup_{\alpha\in\mathbb{R}^{\times}}\left\{ \Pi(S^J+\alpha R^J)\right\}\right)
		\bigcup \left(\bigcup_{\alpha\in\mathbb{R}^{\times}}\left\{ \Pi(\alpha(S^J+ P^J))\right\}\right).\nonumber
	\end{eqnarray}
	In particular, there are infinitely many
	nilpotent $G^J$-orbits in ${\mathcal N}_{\mathbb{R}}^J\subset {\mathfrak g}^J.$
\end{theorem}
\begin{proof}
	By (\ref{4eq07}), It is easily checked that
	\begin{eqnarray}
	{\Cal N}_{\BR}^{J}&=&\left\{ (X,(p,q,r))\,|\ X\in {\Cal N}_{\BR},\ p,q,r\in
	\BR\,\right\}, \nonumber \\
	&=&{\Cal N}_{\BR}^{J,0}\;\bigcup\;{\Cal N}_{\BR}^{J,+}\;\bigcup\;{\Cal N}_{\BR}^{J,-}.
	\end{eqnarray}
Here
\begin{eqnarray*}
{\Cal N}_{\BR}^{J,0}&=&
\left\{ (X,(p,q,r))\,|\ X\in \{0\} ,\ p,q,r\in
\BR\,\right\},	\\
{\Cal N}_{\BR}^{J,+}&=&
\left\{ (X,(p,q,r))\,|\ X\in {\Cal N}_{\BR}^+ ,\ p,q,r\in
\BR\,\right\},	\\
{\Cal N}_{\BR}^{J,-}&=&
\left\{ (X,(p,q,r))\,|\ X\in {\Cal N}_{\BR}^- ,\ p,q,r\in
\BR\,\right\}.	
\end{eqnarray*}
According to Lemma \ref{4lem03}\,--\,\ref{4lem07},
	\begin{eqnarray}
{\Cal N}_{\BR}^{J,0}&=&\bigcup_{p,q,r\in\mathbb{R}}\Pi(G(0,0,0,p,q,r)), \nonumber\\
&=&\Pi(P^J)\bigcup\left\{0\right\}\bigcup\left(\bigcup_{\al \in\BR^{\times}}\left\{
\Pi(\al R^J)\right\}\right).	\label{4eq22}\\
{\Cal N}_{\BR}^{J,+}\;\bigcup\;{\Cal N}_{\BR}^{J,-}&=&\bigcup_{\substack{\alpha \in  \mathbb{R}^{\times} \\ p,q,r\in\mathbb{R}}}\Pi(\alpha G(0,1,1,p,q,r)),\nonumber\\
 &=&\bigcup_{\substack{\alpha \in \{{\pm 1}\} \\ p,r\in\mathbb{R}}}(\Pi(\alpha S^J+pP^J)+\Pi(r R^J)),\nonumber\\
 & = &\Pi(S^J)\bigcup \Pi(T^J)\bigcup \left(\bigcup_{\alpha\in\mathbb{R}^{\times}}\left\{ \Pi(S^J+\alpha R^J)\right\}\right)\label{4eq23} \\
 & & \bigcup \left(\bigcup_{\alpha\in\mathbb{R}^{\times}}\left\{ \Pi(\alpha(S^J+ P^J))\right\}\right).\nonumber
	\end{eqnarray}
Clearly, the set in (\ref{4eq22}) and the set in (\ref{4eq23}) are disjoint union.
Hence, we obtain the formula (\ref{4eq21}).
\end{proof}
For $x,y,p,q\in \BC,$ we
set
\begin{equation}
H(x,y,p,q):=\begin{pmatrix} x & 0 & \ y & \ q \\ p & 0 & \ q & \ 0\\ y &
0 & -x & -p \\ 0 & 0 & \ 0 & \ 0\end{pmatrix}.
\end{equation}
Let $\fpc^J$ be
the vector space consisting of all
$H(x,y,p,q)\,(x,y,p,q\in\BC)$. Let $\fg_{\BC}^J$ be
the complexification of $\fgj$. Then we have the direct sum
\begin{equation}\label{4eq27}
\fg_{\BC}^J=\fk_{\BC}^J+\fp_{\BC}^J,
\end{equation}
where
$$\fk_{\BC}^J=\left\{ \left( \begin{pmatrix} 0 & x \\ -x & 0 \end{pmatrix},(0,0,\kappa)\right)  \bigg|\
x,\kappa\in\BC\,\right\}$$
is the complexification of the Lie algebra
$\fk^J$ of $K^J$. Let ${\Cal N}_{\theta}^J$ be the set of all
nilpotent elements in $\fp_{\BC}^J$. Then
\begin{equation}
{\Cal N}_{\theta}^J=\left\{ H(x,y,p,q)\in\fp_{\BC}^J\;\vert\
x^2+y^2=0\,\right\}.
\end{equation}
Indeed, (\ref{4eq27}) follows from the
fact that
\begin{equation}
H(x,y,p,q)^{2k}=(x^2+y^2)^{k-1}H(x,y,p,q)^2\quad\text{for\
all}\ k\in\BZ^+.
\end{equation}

We set
\begin{eqnarray*}
X_{\theta}^J&=&{\frac 12}\,H(-i,1,0,0),\\
Y_{\theta}^J&=&{\frac 12}\,H(i,1,0,0),\\
P_{\theta}^J&=&H(0,0,1,0),\\
Q_{\theta}^J&=&H(0,0,0,1).
\end{eqnarray*}

\vskip 0.35cm\noindent
Obviously the set
$\left\{X_{\theta}^J,\,Y_{\theta}^J,\,P_{\theta}^J,\,Q_{\theta}^J\right\}$
forms a basis for a complex vector space $\fp_{\BC}^J.$

\vskip 0.3cm
\noindent
\begin{proposition} $K_{\BC}^J$ acts on $\fp_{\BC}^J$
preserving ${\Cal N}_{\theta}^J$.
\end{proposition}
\begin{proof}
	An element $k^J$ of $K_{\BC}^J$ is of the form
	\begin{equation}
	k^J=\begin{pmatrix} a & 0 & b & 0 \\ 0 & 1 & 0 & \kappa
	\\ -b & 0 & a & 0 \\ 0 & 0 & 0 & 1\end{pmatrix},\quad a,b,\kappa \in\BC,\
	a^2+b^2=1.
	\end{equation}
	We obtain
	$$k^J\cdot
	H(x,y,p,q)=H(x^{\ast},y^{\ast},p^{\ast},q^{\ast}),$$
	where
	\begin{eqnarray}\label{4eq30}
	x^{\ast}&=&(a^2-b^2)x+2aby,\nonumber\\
	y^{\ast}&=&-2abx+(a^2-b^2)y,\nonumber\\
	p^{\ast}&=&ap+bq,\\
	q^{\ast}&=& aq-bp.\nonumber
	\end{eqnarray}
	If $H(x,y,p,q)\in
	{\Cal N}_{\theta}^J,$ then
	$$(x^{\ast})^2+(y^{\ast})^2=(a^2+b^2)^2(x^2+y^2)=0.$$
	Therefore
	$k^J\cdot {\Cal N}_{\theta}^J\subset {\Cal N}_{\theta}^J$ for each
	element $k^J\in K_{\BC}^J.$
\end{proof}

\vskip 0.5cm
We define the $K_{\BC}^J$-orbits ${\Cal N}_{\theta}^{J,+}$ and
${\Cal N}_{\theta}^{J,-}$ by
$${\Cal N}_{\theta}^{J,+}=K_{\BC}^J\cdot X_{\theta}^J\quad\text{and}\quad
{\Cal N}_{\theta}^{J,-}=K_{\BC}^J\cdot Y_{\theta}^J.$$
It is easily checked that ${\Cal N}_{\theta}^{J,+}$ and
${\Cal N}_{\theta}^{J,-}$ are given by
\begin{eqnarray} \label{4eq32}
{\Cal N}_{\theta}^{J,+}&=&\left\lbrace H(x,y,0,0) \in {\mathfrak g}_{\mathbb{C}}^J \;\Bigg|\; \begin{array}{c}
x=-i  a^{2} + 2  a b + i  b^{2},\; y= a^{2} + 2 i  a b - b^{2}\nonumber\\
a^2+b^2=1, \,  a,b\in \mathbb{C}
\end{array}  \right\rbrace , \\
						&=&\left\{H(x,ix,0,0)\,\vert\ x\neq 0,\ x\in\BC\,\right\},
\end{eqnarray}
and
\begin{eqnarray}
{\Cal N}_{\theta}^{J,-}&=&\left\lbrace H(x,y,0,0) \in {\mathfrak g}_{\mathbb{C}}^J \;\Bigg|\; \begin{array}{c}
x=i  a^{2} + 2 a b - i b^{2},\; y= a^{2} - 2 i a b - b^{2} \nonumber\\
a^2+b^2=1, \,  a,b\in \mathbb{C}
\end{array}  \right\rbrace , \\
&=&\left\{H(x,-ix,0,0)\,\vert\ x\neq 0,\ x\in\BC\,\right\}.
\end{eqnarray}
We define, for a nonzero
complex number $\delta\in\BC$,
$${\Cal N}_{\theta}^{J,P}(\delta)=K_{\BC}^J\cdot \delta
P_{\theta}^J\quad\text{and}\quad {\Cal
N}_{\theta}^{J,Q}(\delta)=K_{\BC}^J\cdot \delta Q_{\theta}^J.$$
Then
\begin{equation}
{\Cal N}_{\theta}^{J,P}(\delta)={\Cal
N}_{\theta}^{J,Q}(\delta)=\left\{ H(0,0,p,q)\;\vert\
p^2+q^2=\delta^2,\ p,q\in\BC\,\right\}.
\end{equation}
\begin{lemma}\label{4lem05}
$H(x,y,\delta,0)$ and $H({\tilde x},{\tilde
y},\delta,0)$ lie in the same $K_{\BC}$-orbit in $\fpc^J$ if and
only if ${\tilde x}=x,\ {\tilde y}=y.$
\end{lemma}
\begin{proof}
By (\ref{4eq30}), we have $a=1$ and $b=0.$ Hence ${\tilde x}=x,\ {\tilde
	y}=y.$
\end{proof}
\begin{lemma}
Let $(x,y)\in \BC^2$ with $(x,y)\neq (0,0).$
Then $H(x,y,\delta,0)$ and $H(x,y,{\tilde{\delta}},0)$ lie in the same $K_{\BC}^J$-orbit in $\fpc^J$ if and only if
${\tilde{\delta}}=\pm \delta.$
\end{lemma}
\begin{proof}
	According to (\ref{4eq30}), $b=0$ and so $a=\pm 1$.
	Since ${\tilde{\delta}}=a \delta,\ {\tilde{\delta}}=\pm \delta$.
\end{proof}
\begin{lemma}
Suppose $H(x,y,p,q)\in {\Cal N}_{\theta}^J$
with $y=\xi x,$ where $\xi=i$ or $-i$. If $H({\tilde x},{\tilde
y},{\tilde p},{\tilde q})$ is in the $K_{\BC}$-orbit of
$H(x,y,p,q)$ in $\fpc^J$, then ${\tilde y}=\xi{\tilde x}.$
\end{lemma}
\begin{proof}
According to (\ref{4eq30}), we obtain
$${\tilde x}=(a+\xi b)^2x,\quad {\tilde y}=\xi (a+\xi b)^2x.$$
Hence ${\tilde y}=\xi{\tilde x}.$
\end{proof}

\begin{lemma}\label{4lem08}
Let $x,\delta\in \BC$ with $x\neq 0$. We denote
by ${\Cal N}_{\theta}^{J,+}(x,\delta)$ and ${\Cal N}_{\theta}^{J,-}(x,\delta)$ the $K_{\BC}$-orbits of
$H(x,ix,\delta,0)$ and $H(x,-ix,\delta,0)$ respectively. Then
${\Cal N}_{\theta}^{J,+}(x,\delta)$ and ${\Cal
N}_{\theta}^{J,-}(x,\delta)$ are given by
\begin{equation}\label{4eq35}
{\Cal N}_{\theta}^{J,+}(x,\delta)=\left\{ H(z,iz,p,q)\;\vert\ z,p,q\in
\BC,\ p^2+q^2=\delta^2\,\right\}
\end{equation}
and
\begin{equation}\label{4eq36}
{\Cal N}_{\theta}^{J,-}(x,\delta)=\left\{ H(z,-iz,p,q)\;\vert\
z,p,q\in \BC,\ p^2+q^2=\delta^2\,\right\}.
\end{equation}
\end{lemma}
\begin{proof}
	If $H({\tilde x},{\tilde y},{\tilde p},{\tilde q})$ is
	an element of ${\Cal N}_{\theta}^{J,+}(x,\delta)$, then there
	exist $a,b\in\BC$ with $a^2+b^2=1$ satisfying
	$${\tilde x}=(a+ib)^2x,\ {\tilde y}=i(a+ib)^2x,\ {\tilde p}=a\delta,\
	{\tilde q}=-b\delta.$$
	Thus ${\tilde y}=i{\tilde
		x},\ {\tilde p}^2+{\tilde q}^2=(a^2+b^2)\delta^2=\delta^2.$ Hence
	we obtain the formula (\ref{4eq35}). In a similar way, we get the formula
	(\ref{4eq36}).
\end{proof}

\vskip 0.5cm
According to (\ref{4eq32})\,--\,(\ref{4eq36}), Lemma \ref{4lem05}\,--\,Lemma \ref{4lem08}, we obtain the
following theorem.

\vskip 0.5cm
\noindent
\begin{theorem}\label{4thm02}
	We have the following disjoint union
	\begin{eqnarray*}
		{\mathcal{N}}_{\theta}^J&=& \left\{0\right\}\,\bigcup\,{\mathcal{N}}_{\theta}^{J,+}\,\bigcup\, {\mathcal{N}}_{\theta}^{J,-}\,\bigcup\,\left( \bigcup_{\delta\in
			\mathbb{C}^{\times}/\{\pm 1\} }{\mathcal{N}}_{\theta}^{J,P}(\delta)\right)  \\
		& & \ \  \bigcup\left(\bigcup_{\substack{ x\in \mathbb{C}^{\times}\\
				\delta\in {\mathbb{C}}^{\times}/\{\pm 1\} }} {\mathcal{N}}_{\theta}^{J,+}(x,\delta)\right) \bigcup\, \left(
		\bigcup_{\substack{x\in \mathbb{C}^{\times}\\ \delta\in {\mathbb{C}}^{\times}/\{\pm 1\} }} {\mathcal{N}}_{\theta}^{J,-}(x,\delta)\right).\,
	\end{eqnarray*}
	In particular, there are
	infinitely  many nilpotent $K_{\mathbb{C}}^J$-orbits in ${\mathcal{N}}_{\theta}^J\subset {\mathfrak {p}}_{\mathbb{C}}^J.$
\end{theorem}

\vskip 0.5cm\noindent
\begin{remark}
It is known that if $G$ is a real reductive
Lie group, there are only finitely many nilpotent orbits and that
there is the so-called Kostant-Sekiguchi correspondence between
the set of all nilpotent $G$-orbits in $\fg$ and the set of all
nilpotent $K_{\BC}$-orbits in $\fpc,$ where $K_{\BC}$ is the
complexification of a maximal compact subgroup $K$ of $G$ and
$\fgc=\fkc+\fpc$ is the Cartan decomposition of the
complexification $\fgc$ of $\fg$\,(cf.\,\cite{S-V, Se, Ve, Vo}).
We refer to \cite{C} for adjoint orbits of complex semisimple groups.
\end{remark}

\end{section}

\vskip 1.5cm

\vskip 1.5cm


\begin{thebibliography}{99}

\bibitem{B} S. Berceanu, {\em Balanced metric and Berezin quantization on the Siegel-Jacobi ball}, SIGMA Symmetry Integrability Geom. Methods Appl. {\bf 12} (2016), Paper No. 064, 28 pp.

\bibitem{BS} R. Berndt and R. Schmidt, {\em Elements of the
Representation Theory of the Jacobi Group}, Progress in
Mathematics, {\bf 163}, Birkh{\"a}user, Basel, 1998.

\bibitem{C} P. Crooks, {\em Complex adjoint orbits in Lie theory and geometry}, arXiv:1703.03390v1 [math.AG] 9 Mar 2017.

\bibitem{EZ} M. Eichler and D. Zagier, {\em The Theory of Jacobi Forms}, Progress in Mathematics {\bf 55}, Birkh{\"a}user, Boston, Basel and Stuttgart, 1985.

\bibitem{K} D. King. {\em Spherical nilpotent orbits and the Kostant-Sekiguchi Correspondence}, Trans. Amer. Math. Soc. {\bf 354.12}(2002), 4909-4920.

\bibitem{KR} B.Kostant and S. Rallis, {\em Orbits and Representations Associated with Symmetric Spaces}, American Journal of Mathematics,{\bf 93} (1971) no. 3, 753-809.
\bibitem{L} S. Lang, {\em $SL_2(\BR)$}, Springer-Verlag, Berlin-Heidelberg-New York (the second edition)
(1985).

\bibitem{M} M. Molitor, {\em Gaussian distributions, Jacobi group, and Siegel-Jacobi space}, J. Math. Phys.  {\bf 55}  (2014),  no. 12, 122102, 40 pp.


\bibitem{S-V} W. Schmid and K. Vilonen, {\em On the geometry of nilpotent orbits}, Asian J. Math. {\bf 3} (1999),
no. 1, 233-274.

\bibitem{Se} J. Sekiguchi, {\em Remarks on nilpotent orbits of a symmetric pair},
Jour. Math. Soc. Japan {\bf 39} (1987), 127-138.


\bibitem{Ve} M. Vergne,  {\em Instantons et correspondence de Kostant-Sekiguchi}, C.R. Acad. Sci. Paris {\bf 320} (1995), 901-906.

\bibitem{Vo} D. Vogan, {\em The Methods of Coadjoint Orbits for Real Reductive Groups in Representation Theory of Lie Groups},
IAS/PARK CITY, Math. Series, {\bf vol. 8}, AMS and IAS (2000), 179-238.


\bibitem{YJH1} J.-H. Yang, {\em Remarks on Jacobi forms of higher degree}, Proc. of the 1993 Workshop on Automorphic Forms and Related Topics, edited by Jin-Woo Son and Jae-Hyun Yang, the Pyungsan Institute for Mathematical Sciences (1993), 33--58.

\bibitem{YJH2} J.-H. Yang, {\em Vanishing theorems on Jacobi forms of higher degree},
J. Korean Math. Soc., {\bf 30}(1)(1993), 185--198.



\bibitem{YJH3}  J.-H. Yang, {\em The Siegel-Jacobi Operator}, Abh. Math. Sem. Univ. Hamburg {\bf 63} (1993), 135--146.

\bibitem{YJH4} J.-H. Yang, {\em Singular Jacobi Forms,} Trans. Amer. Math. Soc. {\bf 347\,(6)} (1995), 2041--2049.

\bibitem{YJH5} J.-H. Yang, {\em Construction of vector valued modular forms from Jacobi forms,} Canadian J. of Math. {\bf 47\,(6)} (1995), 1329--1339.



\bibitem{YJH6} J.-H. Yang, {\em Kac-Moody Algebras, the Monstrous Moonshine, Jacobi Forms and Infinite Products}, Proceedings of the 1995 Symposium on Number Theory, Geometry and Related Topics, edited by Jin-Woo Son and Jae-Hyun Yang, the Pyungsan Institute for Mathematical Sciences (May 1996), 13--82. MR1404967 {\it or} arXiv:math/0612474v3 [math.NT] 23 May 2017.

\bibitem{YJH7} J.-H. Yang, {\em Stable Automorphic Forms}, Proceedings of Japan-Korea Joint Seminar on Transcendental Number Theory and Related Topics~(1998), 101--126.

\bibitem{YJH8} J.-H. Yang, {\em A geometrical theory of Jacobi forms of higher degree,} Proceedings of Symposium on Hodge Theory and Algebraic Geometry\,(\,edited by Tadao Oda\,), Sendai, Japan (1996), 125--147 {\it or} Kyungpook Math. J. {\bf 40\,(2)} (2000), 209--237 {\it or } arXiv:math/0602267v1 [math.NT] 13 Feb 2006.

\bibitem{YJH8.5} J.-H. Yang, {\em  The Method of Orbits for Real Lie Groups,} Kyungpook Math. J. {\bf 42\,(2)} (2002), 199-272 {\it or } arXiv:math/0602056v1 [math.RT] 3 Feb 2006.

\bibitem{YJH9} J.-H. Yang, {\em A note on a fundamental domain for Siegel-Jacobi space,} Houston Journal of Mathematics {\bf 32\,(3)} (2006), 701--712.

\bibitem{YJH10} J.-H. Yang, {\em Invariant metrics and Laplacians on Siegel-Jacobi space,} Journal of Number Theory {\bf 127} (2007), 83--102.

\bibitem{YJH11} J.-H. Yang, {\em A partial Cayley transform for Siegel-Jacobi disk,} J. Korean Math. Soc. {\bf 45}, No.~3 (2008), 781--794.



\bibitem{YJH12} J.-H. Yang, {\em Invariant metrics and Laplacians on Siegel-Jacobi disk,} Chinese Annals of Mathematics, Vol. 31B(1), 2010, 85--100.



\bibitem{YJH13} J.-H. Yang, Y.-H. Yong, S.-N. Huh, J.-H. Shin and G.-H. Min, {\em Sectional Curvatures of the Siegel-Jacobi Space,} Bull. Korean Math. Soc. {\bf 50} (2013), No.~3, pp. 787--799.



\bibitem{YJH14} J.-H. Yang, {\em A Note on Maass-Jacobi Forms II,} Kyungpook Math. J. {\bf 53}, no.~1 (2013), 49--86.

\bibitem{YJH15} J.-H. Yang, {\em Covariant maps for the Schr{\"o}dinger-Weil representation,} Bull. Korean Math. Soc. {\bf 52} (2015), No.~2, pp. 627--647.


\bibitem{YJH16} J.-H. Yang, \emph{Geometry and Arithmetic on the Siegel-Jacobi Space}, Geometry and Analysis on Manifolds, In Memory of Professor Shoshichi Kobayashi (edited by T. Ochiai, A. Weinstein et al), Progress in Mathematics, Volume {\bf 308}, Birkh{\"a}user, Springer International Publishing AG Switzerland (2015), 275--325 {\it or } arXiv:1702.08663v1 [math.NT] 28 Feb 2017.

\bibitem{YJH17} J.-H. Yang, {\em Theta sums of higher index,}
Bull. Korean Math. Soc. {\bf 53} (2016), No. 6, pp.~1893--1908 {\it or } arXiv:1702.08667 [math.NT] 28 Feb 2017.


\bibitem{YJH18} J.-H. Yang, {\em Stable Schottky-Jacobi forms,} arXiv:1702.08650v1 [math.NT] 28 Feb 2017.

\bibitem{YY} J. Yang and L. Yin, {\em Differetial operators for Siegel-Jacobi forms,} Science China Mathematics, Vol. {\bf 59} (2016), No.~6, pp. 1029--1050.



\bibitem{Z}  C. Ziegler, {\em Jacobi Forms of Higher Degree}, Abh. Math. Sem. Hamburg {\bf 59} (1989), 191--224.

\end{thebibliography}
\end{document}